\documentclass[12pt,fleqn]{article}
\usepackage{amsfonts,amsmath, amsthm,amssymb}

\usepackage{latexsym}
\usepackage{esint}

\usepackage[margin=22mm]{geometry} 
 \numberwithin{equation}{section}

\usepackage{color}
 \definecolor{db}{rgb}{0.0,0.0,0.8} 
\definecolor{dg}{rgb}{0.0,0.55,0.14}
\definecolor{dr}{rgb}{0.5,0,0.07}

  \usepackage{hyperref}
 \hypersetup{frenchlinks=true,colorlinks=true,linkcolor=db,citecolor=dg,
 bookmarks=true}

 \usepackage{enumerate}
\usepackage{authblk}

\newtheorem{theorem}{Theorem}[section]
\newtheorem{proposition}[theorem]{Proposition}

\newtheorem{lemma}[theorem]{Lemma}
\newtheorem{corollary}[theorem]{Corollary}
\theoremstyle{definition}

\theoremstyle{definition}

\theoremstyle{definition}

\theoremstyle{definition}

\theoremstyle{definition}

\theoremstyle{definition}
\newtheorem{remark}[theorem]{Remark}
\theoremstyle{definition}

\newtheorem{open-problem}{Open Problem}
\newcounter{step}

\newcommand{\rlemma}[1]{Lemma~\ref{#1}}
\newcommand{\rth}[1]{Theorem~\ref{#1}}

\newcommand{\rprop}[1]{Proposition~\ref{#1}}

\newcommand{\tD}{{\widetilde D}}
\newcommand{\tC}{{\widetilde C}}
\newcommand\blfootnote[1]{%
  \begingroup
  \renewcommand\thefootnote{}\footnote{#1}%
  \addtocounter{footnote}{-1}%
  \endgroup
}

\def\be{\begin{equation}}
\def\ee{\end{equation}}
\def\bes{\begin{equation*}}
\def\ees{\end{equation*}}
\def\bt{\begin{theorem}}
\def\et{\end{theorem}}
\def\bpr{\begin{proposition}}
\def\epr{\end{proposition}}
\def\bl{\begin{lemma}}
\def\el{\end{lemma}}
\def\bc{\begin{corollary}}
\def\ec{\end{corollary}}
\def\br{\begin{remark}}
\def\er{\end{remark}}
\def\ben{\begin{enumerate}}
\def\bena{\begin{enumerate}[a)]}
\def\een{\end{enumerate}}
\def\bit{\begin{itemize}}
\def\iit{\end{itemize}}

\def\deg{\operatorname{deg}}

\DeclareMathAlphabet{\mathonebb}{U}{bbold}{m}{n}

\def\R{{\mathbb R}}

\def\C{{\mathbb C}}

\usepackage{csquotes}

\def\ve{\varepsilon}

\newcommand{\wtRz}{\widetilde{R}_0}
\newcommand{\wtRo}{\widetilde{R}_1}
\newcommand{\wtC}{\widetilde{C}}

\title{Small energy Ginzburg-Landau minimizers in $\R^3$}
\author{Etienne Sandier\thanks{LAMA-CNRS UMR 8050, Universit\'e
    Paris-Est, 61 Avenue du G\'en\'eral De Gaulle, 94010 Cr\'eteil,
    France. 
    Email address: sandier@u-pec.fr}
      \and
 Itai Shafrir \thanks{Department of Mathematics, Technion - I.I.T., 32000 Haifa, Israel.\\ Email address: shafrir$@$math.technion.ac.il}
}

\begin{document}
\maketitle
\begin{abstract}
We prove that a local minimizer of the Ginzburg-Landau energy in
$\R^3$ 
satisfying the condition $\liminf_{R\to\infty}\frac{E(u;B_R)}{R\ln
  R}<2\pi$ must be constant. The main tool is a new sharp
$\eta$-ellipticity result for minimizers in dimension three that might
be of independent interest.  
\end{abstract}

\section{Introduction and main results}
 Consider the Ginzburg-Landau equation in $\R^N$
 \begin{equation}
   \label{eq:81}
   -\Delta u=(1-|u|^2)u \,.
 \end{equation}
Much effort has been devoted to classifying the solutions to
\eqref{eq:81} under various assumptions. In the scalar case, the
famous De Giorgi conjecture states that any bounded solution satisfying
$\partial_{x_N}u>0$ must depend on one Euclidean variable only, at
least when $N\leq
8$. This conjecture was proved to be true in dimension $N=2$ by
Ghoussoub and Gui~\cite{gg}, for $N=3$ by Ambrosio and
Cabr\'e~\cite{ac}, and under the additional assumption
$\lim_{x_N\to\pm\infty} u(x',x_N)=\pm 1$, for $4\leq N\leq 8$, by
Savin~\cite{savin}. On the other hand,  
the counterexample of del Pino et al.~\cite{dp}  shows that indeed $N=8$
is optimal. 

Much less is known in the vector-valued case $u:\R^N\to\R^m$. In that case the
monotony hypothesis no longer makes sense. On the other hand, the
class of locally minimizing solutions in the sense of De Giorgi --- i.e., solutions  that are
minimizing for their own boundary values on every ball --- come up naturally as  blow-up
limits of minimizers of the Ginzburg-Landau energy
\begin{equation*}
  E_\varepsilon(u)=\int \frac{1}{2}|\nabla u|^2+\frac{1}{4\varepsilon^2}(1-|u|^2)^2\,,
\end{equation*}
as $\varepsilon$ goes to zero. In fact, monotone scalar solutions of
\eqref{eq:81} also have a certain local minimality property (see
\cite[Thm~4.4]{aac}).

When $N=m=2$ it follows from the results of Sandier~\cite{s} and Mironescu~\cite{m} that
every local minimizer is either constant or equal to $U(x):=f(|x|)\frac{x}{|x|}$, up to
rotations and translations, where $f$ is the unique solution of the
corresponding ODE. When $N=m=3$ a similar classification was
proved by Millot and Pisante~\cite{mp} under the assumption 
\begin{equation*}
  \limsup_{R\to\infty}\frac{E(u;B_R)}{R}<\infty\quad(\text{here }E(u;B_R)=E_1(u;B_R))\,.
\end{equation*}
Under additional assumptions,  Pisante~\cite{p} extended the result to $N=m\geq 4$.
Note that in these cases the existence of a non-constant local
minimizer is a nontrivial fact. 

We are interested in the case $N=3$, $m=2$. In this case, it is easy
to deduce from the local minimality property of $U$ that
$V(x,z)=U(x)$ is a local minimizer. We conjecture that it is the only
non-constant 
one, up to the obvious symmetries of the problem. As a first step in
this direction we prove:
\begin{theorem}
\label{th:lm}
  If $u:\R^3\to\R^2$ is a local minimizer of the Ginzburg-Landau
  energy such that
\begin{equation*}
  \liminf_{R\to\infty}\frac{E(u;B_R)}{R\ln R}<2\pi\,,
\end{equation*}
 then $u$ is constant.
\end{theorem}
Note that the constant $2\pi$ is optimal since
\begin{equation*}
  \lim_{R\to\infty}\frac{E(V;B_R)}{R\ln R}=2\pi\,.
\end{equation*}
A different assumption was considered by Farina~\cite{f} who proved  that
a local minimizer $u:\R^N\to\R^2$ with $N=3$ or $4$, is constant
provided $\lim_{|x|\to\infty} |u(x)|=1$. 
 
\rth{th:lm} is an easy consequence of the following sharp
$\eta$-ellipticity type result for minimizers in dimension three:
\begin{theorem}
  \label{th:eta}
For every $\gamma\in(0,2\pi)$ and $\lambda>0$  there exists
$\varepsilon_1(\gamma,\lambda)$ such that for every $u_\varepsilon$ which is a minimizer of $E_\varepsilon$ on $B_1$
with $\varepsilon\leq \varepsilon_1(\gamma,\lambda)$ satisfying $E_\ve(u;B_1)\leq \gamma |\ln \ve|$
there holds
\begin{equation}
\label{eq:60}
  \big||u_\varepsilon(0)|-1\big|\leq \lambda\,.
\end{equation}
\end{theorem}
Such a result  first appeared in  
the work of Rivi\`ere~\cite{r} for the case
of minimizers in dimension three. There were subsequent
generalizations by  Lin-Rivi\`ere~\cite{lr1,lr2}
and Bethuel-Brezis-Orlandi~\cite{bbo1,bbo}. The result in \cite{bbo}
is the most general, covering the case   of solutions (not necessary
minimizers) 
of the Ginzburg-Landau equation in any dimension.  

All these results establish the existence of a constant $\gamma>0$
for which the result is true, but no explicit bound is
given. We are able to give the optimal bound, but only for minimizers
in dimension three. Working with minimizers allows us to apply a
construction of an appropriate test function. This is done in
\rprop{prop} below, which plays an important role in the proof of
\rth{th:eta} (see Section~\ref{sec:notation} for notation).   
\begin{proposition}
\label{prop}
 Let $u$ be a minimizer for $E$ on $B_R\subset\R^3$, for its 
 boundary values on $S_R$.
  If 
  \begin{equation}
\label{eq:ineq-prop}
   E^{(T)}(u;S_R)  \leq\gamma\ln R\,,
  \end{equation}
with $\gamma<2\pi$ and
  $R>r_0(\gamma)$ then there exist 
  $\alpha=\alpha(\gamma)\in(0,1)$ and a universal $\sigma\in(0,1)$ such that 
  \begin{equation}
\label{eq:prop}
    E(u;B_R)\leq \sigma R E^{(T)}(u;S_R)  +C_\gamma R^\alpha\ln R\,.
  \end{equation}
\end{proposition}
\subsubsection*{Acknowledgments} The research was carried out during
several mutual visits at the Technion and Universit\'e
    Paris-Est. The authors acknowledge the hospitality and support of
    these two institutions.

\section{Notation and some preliminary  results}
\label{sec:notation}
We begin by introducing some notation.
By $B_R(a)$ we denote a ball in $\R^N$, $N\geq 2$ (usually for $N=3$) and then $S_R(a)=\partial
B_R(a)$.  Specifically in dimension two,  we denote by $D_R(a)$ a disc
and by
$C_R(a)=\partial D_R(a)$ its boundary. In case $a=0$ we denote for short:
\begin{equation*}
  B_R=B_R(0)\,,~S_R=\partial B_R(0)\,,~D_R=D_R(0)\,,~C_R=\partial D_R(0)\,.
\end{equation*}
 For a set $D$ in $\R^N$, $N\geq 2$, $\varepsilon>0$ and $u\in H^1(D;\R^2)$ we denote:
\begin{equation*}
  E_{\ve}(u;D)=\int_D \frac{1}{2}|\nabla
  u|^2+\frac{1}{4\varepsilon^2}(1-|u|^2)^2\,.
\end{equation*}
For $\tD\subseteq S_R(a)\subset\R^N, N\geq 2$, and $u\in H^1(\tD;\R^2)$ we denote
\begin{equation*} 
E_{\ve}^{(T)}(u;\tD)=\int_{\tD} \frac{1}{2}|\nabla_T u|^2+\frac{1}{4\varepsilon^2}(1-|u|^2)^2\,.
\end{equation*}
When $\varepsilon=1$ we denote for short $E(u;D)=E_1(u;D)$ and $E^{(T)}(u;\tD)=E_1^{(T)}(u;\tD)$.

We denote by ${\widetilde
  D}_\rho^R(a)$ a spherical disc on the sphere $S_R=S_R(0)\subset\R^3$, with center at $a$ and radius
$\rho$ (using the geodesic distance) and by ${\widetilde
  C}_\rho^R(a)=\partial {\widetilde  D}_\rho^R(a)$
its boundary. In case there is no risk of
confusion we shall omit the superscript $R$. We often identify
$\R^2$-valued maps with $\C$-valued maps. \\
 We recall the following basic estimates valid for an arbitrary (vector
 valued) solution
 of the Ginzburg-Landau equation \eqref{eq:81}
 in dimension $N\geq2$.
 The first is a $L^\infty$-bound for $u$ and its gradient (\cite[Lemma~
 III.2]{bbo}).
 \begin{lemma}
   \label{lem:L-infty}
Assume $u$ satisfies \eqref{eq:81} in $B_1\subset\R^N$.
Then, there is a  constant $K=K(N)>0$ such that
\begin{align}
  \label{eq:58a}
  \|u\|_{L^\infty(B_{1/2})}\leq K\,,\\
\intertext{and}
  \label{eq:58b}
\|\nabla u\|_{L^\infty(B_{1/2})}\leq K\,.
\end{align}
 \end{lemma}
 \begin{proof}[Sketch of proof]
   For the proof of \eqref{eq:58a} we argue as in
   Brezis~\cite[Remark~3]{brezis} to get, using Kato's inequality,  that $\varphi=(|u|^2-1)^+$
   satisfies 
   \begin{equation}
     \label{eq:58}
     \Delta\varphi\geq\varphi^2~\text{ in }B_R.
   \end{equation}
The result then
   follows by the Keller-Osserman theory (see \cite{brezis2} and the
   references therein).

 Once \eqref{eq:58a} is established, \eqref{eq:58b} follows by standard
 elliptic estimates.
 \end{proof}
The second is a version of the monotonicity formula
(\cite[Corollary~II.1]{bbo}):
\begin{lemma}
  \label{lem:mono}
Let $u$ be a solution of \eqref{eq:81} in $B_R\subset\R^N$. Then, the function $r
\mapsto E(u;B_r)/r^{N-2}$ is nondecreasing in $(0,R)$.
\end{lemma}
Another useful result is the following estimate for harmonic functions
in balls.
\begin{lemma}
Let $w\in C^1(\overline{B}_R)$ be a harmonic function in a ball $B_R\subset\R^N$. Then
\begin{equation}
\label{eq:harmonic}
\int_{B_R} |\nabla w|^2\leq \frac{R}{N-1}\int_{\partial B_R} |\nabla_T w|^2\,.
\end{equation}
\end{lemma}
\begin{proof}
First note that since $|\nabla w|^2$ is subharmonic we have
\begin{equation*}
\int_{B_R} |\nabla w|^2\leq\int_0^R \left(\frac{r^{N-1}}{R^{N-1}}\int_{\partial B_R} |\nabla w|^2\right)\,dr=\frac{R}{N}\int_{\partial B_R} |\nabla w|^2\,,
\end{equation*}
i.e., 
\begin{equation}
\label{eq:A1}
N\int_{B_R} |\nabla w|^2\leq R\int_{\partial B_R} \Big(|\nabla_T w|^2+\big|\frac{\partial w}{\partial n}\big|^2\Big)\,.
\end{equation}
On the other hand, Pohozaev identity gives
\begin{equation}
\label{eq:A2}
(N-2)\int_{B_R} |\nabla w|^2=R\int_{\partial B_R} \Big(|\nabla_T w|^2-\big|\frac{\partial w}{\partial n}\big|^2\Big)\,.
\end{equation}
Adding \eqref{eq:A1} with \eqref{eq:A2} yields \eqref{eq:harmonic}.
\end{proof}
\section{Proof of \rprop{prop}}
The next lemma deals with a Ginzburg-Landau minimization problem on a
spherical disc. The proof requires simple modification of the methods
developed for the case of the usual Ginzburg-Landau energy in the
plane with zero degree boundary condition (see \cite{bbh1}).  
\begin{lemma}
  \label{lem:ext} 
Let $u\in H^1(\tD_\rho^R(a),\R^2)$ with
\begin{align}
  \label{eq:1}
  \rho/R<1/10\\
\intertext{and}
 \label{eq:tu-energy}
 E^{(T)}(u;\tC_\rho^R(a))\leq c_1/\rho\,.
 \end{align}
Then, if $\rho\geq R_0(c_1)$ we have 
\begin{equation}
  \label{eq:2}
  \big||u|-1\big|\leq 1/8~\text{ on }\tC_\rho^R(a)\,.
\end{equation}
If we further assume that 
 \begin{equation}
 \label{eq:tu-degre}
\deg(u,\tC_\rho^R(a)) = 0\,,
\end{equation}
then for $\rho\geq R_1(c_1)\geq R_0(c_1) $, any minimizer $v$ of the energy
 $E(\cdot;\tD_\rho^R(a))$ for the boundary data 
$u$ on $\tC_\rho^R(a)$, satisfies:
\begin{equation}
   \big|1-|v|\big|\leq 1/8~\text{ on }
 \tD_\rho^R(a)\,.\label{eq:5}
\end{equation}
\end{lemma}
\begin{proof}
The proof of \eqref{eq:2} is easy and standard (see \cite[Lemma
1]{lin}). Indeed, assume that for some point
$x_0\in \tC_\rho=\tC_\rho^R(a)$ we have
\begin{equation}
  \label{eq:3}
  \big||u(x_0)|-1\big|>1/8\,.
\end{equation}
Then, for any $x_1\in \tC_\rho$ we have, using H\"older inequality on the arc of
circle  $A(x_0,x_1)\subset \tC_\rho$ and \eqref{eq:tu-energy}:
\begin{equation*}
  \big|u(x_0)-u(x_1)\big|\leq \Big(2 E^{(T)}(u;\tC_\rho)\Big)^{1/2}|A(x_0,x_1)|^{1/2}\leq\Big(\frac{2c_1}{\rho}\Big)^{1/2}|A(x_0,x_1)|^{1/2}\,.
\end{equation*}
It follows, using \eqref{eq:3},  that there exists $\lambda>0$ such that
\begin{equation}
\label{eq:69}
  \big||u|-1\big|\geq 1/16~\text{ on }\{|x-x_0|<\lambda \rho/ c_1\}\cap \tC_\rho\,.
\end{equation}
 Clearly, \eqref{eq:69} implies, for some positive constant $c_2$, that
 \begin{equation*}
   \int_{\tC_\rho} (1-|u|^2)^2\geq c_2\rho\,,
 \end{equation*}
which contradicts \eqref{eq:tu-energy} for $\rho$ large
enough. 

For the proof of \eqref{eq:5} it is convenient to treat an equivalent problem for a
certain weighted Ginzburg-Landau energy on a planar disc. For that
matter we will perform a change of variables in several steps. Without loss of generality we may assume that $a=(0,\ldots,0,-R)$, the
south pole of $S_R$. Denoting by $\mathcal{S}:S_1\to\R^2$ the standard 
stereographic projection 
we define a function $\tilde u:D_{\tan(\rho/2R)}\to\R^2$ by $\tilde u(x)=u(R\mathcal{S}^{-1}x)$.
We have 
\begin{equation*}
  E(u; \tD_\rho^R(a))=
\int_{D_{\tan(\rho/2R)}} \Big\{\frac{1}{2}|\nabla\tilde
u|^2+\frac{R^2(1-|\tilde u|^2)^2}{(1+|x|^2)^2}\Big\}dx\,.
\end{equation*}
 A final rescaling, setting $U(y)=\tilde u(\tan(\rho/2R)y)$, yields
 $U:D_1\to\R^2$ satisfying 
\begin{equation*}
  F_\varepsilon(U;D_1)=E(u; \tD_\rho^R(a))
\end{equation*}
with
\begin{equation*}
  F_\varepsilon(U;D_1):=\int_{D_1} \Big\{\frac{1}{2}|\nabla U|^2+\frac{p(y)}{4\varepsilon^2}(1-|U|^2)^2\Big\}\,dy\,,
\end{equation*}
where 
$$
\varepsilon=1/(2R\tan(\rho/2R))~\text{ and
}~p(y)=1/(1+\tan^2(\rho/2R)|y|^2)^2\,.
$$
Note that by \eqref{eq:1} $p(y)$ is bounded between two positive
constants, and all its derivatives are uniformly bounded as well.
Moreover, by \eqref{eq:tu-energy} we have
\begin{equation*}
  \int_{\partial D_1} \Big\{\frac{1}{2}|\nabla
  U|^2+\frac{p(y)}{4\varepsilon^2}(1-|U|^2)^2\Big\}\leq \tilde c_1\,,
\end{equation*}
where $\tilde c_1$ depends on $c_1$ only.

The proof of \eqref{eq:5} follows by a simple modification of the
arguments in \cite[Thm~2]{bbh1} (in particular the proof of $(95)$
there), see also \cite[Lemma 2]{lin}; the fact that here we deal with
a weighted Ginzburg-Landau energy causes no difficulty since the
weight is smooth, as explained above.
 
\end{proof}
The next lemma is concerned with an extension problem in a
cylinder. It will be useful for a similar problem on a spherical
cylinder in the course of the proof of \rprop{prop}.
\begin{lemma}
  \label{lem:cyl}
 Let $u,v\in H^1(D_R,\R^2)$ be such that
 \begin{align}
   \label{eq:85}
   u&=v \text{ on }C_R\,,\\
\big|1-|v|\big|&\leq 1/8\text{ and }|u|\leq K~\text{ on }\overline{D}_R\,.\label{eq:86}
 \end{align}
  Then, for any
$H>0$ there
exists $U\in H^1(D_R\times(0,H))$ satisfying
\begin{align}
  U(x,H)&=u(x),\quad x\in D_R\label{eq:7}\\
 U(x,0)&=v(x),\quad x\in D_R\label{eq:8}\\
 U(x,z)&=u(x)=v(x),\quad x\in\partial D_R,\,z\in(0,H)
\intertext{and}
E(U;D_R\times(0,H))&\leq C(H+R^2/H)\,(E(u;D_R)+E(v;D_R))\,.\label{eq:49}
\end{align}
\end{lemma}

\begin{proof}
  We first extend $v$ to the cylinder $\overline{D}_R\times[0,H]$ by letting 
  \begin{equation}
    \label{eq:32}
    V(x,z)=v(x)\,,\quad x\in \overline{D}_R,~z\in[0,H]\,.
  \end{equation}
Define the cone $\Gamma$ by
\begin{equation}
  \label{eq:34}
 \Gamma=\{(x,z)\,;\,0<|x|<R\,,\,(H/R)|x|<z<H\}\,.
\end{equation}
Let $w$ be defined in $\overline{D}_R$ by $w=u/v$ and then extend
it to $\overline{D}_R\times[0,H]$ by
\begin{equation}
  \label{eq:33}
W(x,z)=\begin{cases}1 & (x,z)\in \overline{D}_R\times[0,H]\setminus\Gamma\\
w(Hx/z) & (x,z)\in\Gamma
\end{cases}\,.
\end{equation}
Finally we set
\begin{equation}
  \label{eq:35}
U=VW\quad\text{ in }\overline{D}_R\times[0,H]\,.
\end{equation}
Clearly,
\begin{equation}
  \label{eq:36}
 E(W;\overline{D}_R\times[0,H])=E(W;\Gamma)=\int_0^Hdz\,\int_{D_{Rz/H}\times\{z\}}
 \frac{1}{2}|\nabla W|^2+\frac{1}{4}(1-|W|^2)^2\,.
\end{equation}
First note that
\begin{equation}
  \label{eq:37}
  \int_{D_{Rz/H}\times\{z\}}|\nabla_x W|^2=\int_{D_{R}}|\nabla w|^2\,,
\end{equation}
 while
 \begin{equation}
   \label{eq:38}
\begin{aligned}
  \int_{D_{Rz/H}\times\{z\}}\Big|\frac{\partial W}{\partial z}\Big|^2&=
\int_{D_{Rz/H}}\big|\nabla w(Hx/z)\cdot(Hx/z^2)\big|^2\\
&\leq
\frac{1}{H^2} \int_{D_R} |\nabla w(y)|^2|y|^2\,dy\leq
\Big(\frac{R}{H}\Big)^2 \int_{D_R} |\nabla w|^2\,.
\end{aligned}
 \end{equation}
Integrating \eqref{eq:37}--\eqref{eq:38} over $z\in(0,H)$ yields
\begin{equation}
  \label{eq:39}
\int_{D_R\times(0,H)}|\nabla W|^2\leq \big(H+R^2/H\big) \int_{D_R} |\nabla w|^2\,. 
\end{equation}
 Since by \eqref{eq:86}
\begin{equation}
\label{eq:uvw}
|\nabla w|\leq C\big(|\nabla u|+|\nabla v|)\,,
\end{equation}
 we deduce from \eqref{eq:39} that
\begin{equation}
 \label{eq:40}
\int_{D_R\times(0,H)}|\nabla W|^2\leq C\big(H+R^2/H\big) \Big(\int_{D_R} |\nabla u|^2+\int_{D_R}|\nabla v|^2\Big)\,. 
\end{equation}
Next we turn to estimate the second term in the energy. We have
\begin{multline}
  \label{eq:41}
\int_{D_{Rz/H}}\big(|W(x,z)|^2-1\big)^2\,dx=\int_{D_{Rz/H}}\big(|w(Hx/z)|^2-1\big)^2\,dx\\=(z/H)^2\int_{D_R}(|w(y)|^2-1)^2\,dy\,,
\end{multline}
so integrating \eqref{eq:41} for $z\in(0,H)$ gives
\begin{equation}
  \label{eq:42}
  \int_\Gamma\big(|W|^2-1\big)^2=(H/3) \int_{D_R} (|w(y)|^2-1)^2\,dy\,.
\end{equation}
Noting that 
\begin{equation}
  \label{eq:46}
 \big||w|^2-1\big|=\frac{\big||u|^2-|v|^2\big|}{|v|^2}\leq (8/7)^2\Big(\big||u|^2-1\big|+\big||v|^2-1\big|\Big)\,,
\end{equation}
we conclude from \eqref{eq:42} that
\begin{equation}
  \label{eq:47}
 \int_\Gamma\big(|W|^2-1\big)^2\leq CH\Big(\int_{D_R} (|u|^2-1)^2+\int_{D_R} (|v|^2-1)^2\Big)\,.
\end{equation}
We also clearly have
\begin{align}
  \int_{D_R\times(0,H)}|\nabla V|^2&=H\int_{D_R}|\nabla v|^2\,,\label{eq:43}\\
\int_{D_R\times(0,H)}(1-|V|^2)^2&=H\int_{D_R} (1-|v|^2)^2\,.\label{eq:44}
\end{align}
Similarly to \eqref{eq:uvw} we have $|\nabla U|\leq c(|\nabla
V|+|\nabla W|)$. Hence, 
from \eqref{eq:40} and \eqref{eq:43} we get that
\begin{equation}
  \label{eq:45}
\int_{D_R\times(0,H)}|\nabla U|^2\leq C\big(H+R^2/H\big) \Big(\int_{D_R} |\nabla u|^2+\int_{D_R}|\nabla v|^2\Big)\,.
\end{equation}
By \eqref{eq:35}, \eqref{eq:86} and \eqref{eq:32}  we have
\begin{equation*}
  \big||U|^2-1\big|\leq
  \big||V|^2|W|^2-|V|^2\big|+\big||V|^2-1\big|\leq (9/8)^2\big||W|^2-1\big|+\big||V|^2-1\big|\,,
\end{equation*}
and applying \eqref{eq:44} and \eqref{eq:47} yields
\begin{equation}
  \label{eq:48}
\int_{D_R\times(0,H)}(1-|U|^2)^2\leq CH\Big(\int_{D_R} (1-|v|^2)^2+\int_{D_R} (1-|u|^2)^2\Big)\,.
\end{equation}
Clearly \eqref{eq:49} follows from \eqref{eq:45}--\eqref{eq:48}.
\end{proof}
\begin{proof}[Proof of \rprop{prop}]
  The proof is divided to several steps.\\[2mm]
\noindent{\underline{Step 1:}} Locating the ``bad discs'' on
$S_R$ and choosing $\alpha$.\\[2mm]

The following result should seem plausible to specialists in the field; we state it as a Lemma and prove it in the appendix. 

\begin{lemma} 
\label{lem:specialists} Given $\gamma < 2\pi$ and $\Lambda>0$ there exists $R_0>0$ such that if $R>R_0$ and $u:S_R\to\R^2$ satisfies 
$$ E^{(T)}(u;S_R)\leq \gamma\ln R,$$
then there exist $\alpha\in(0,1)$ and spherical discs
$\tD_i=\tD_{r_i}^R(a_i)$, $i = 1,\dots, k$, such that 
\begin{align}
 & |u|>7/8 \text{ on }S_R\setminus \bigcup\limits_{i=1}^k\tD_i\,,\label{eq:p1}\\
 &\displaystyle \sum_{i=1}^k r_i \le R^{\alpha}\,,\label{eq:p2}\\
 & E^{(T)}(u,\partial \tD_i) \le \displaystyle\frac{2\pi}{r_i}\,,~i=1,\ldots,k\,, \label{eq:p3}\\
 &\deg(u,\partial \tD_i) = 0\,,~i=1,\ldots,k\,,\label{eq:p4}\\
 &r_i\geq  \Lambda\,,~i=1,\ldots,k\,. \label{eq:p5}
 \end{align}
\end{lemma}
Applying \rlemma{lem:specialists} with 
$\Lambda=R_1(2\pi)$ (see \rlemma{lem:ext}) yields a collection
$\{\tD_{ r_i}^R(a_i)\}_{i=1}^k$ satisfying \eqref{eq:p1}--\eqref{eq:p5}.
\vskip 0.2cm

\noindent{\underline{Step 2:}} Extension to a map in $B_R\setminus
B_{R-R^\alpha}$ with a phase on $S_{R-R^\alpha}$.\\[2mm]
\noindent  For each $i=1,\ldots,k$ we apply \rlemma{lem:ext} with
$c_1=2\pi$ to find a map  $v_i$ defined in $\tD_i=\tD_{ r_i}^R(a_i)$
satisfying:
\begin{flalign}
  &v_i=u~\text{ on }\partial \tD_i\,,\label{eq:18}\\
 &\big||v_i|-1\big|\leq 1/8~\text{ on }\tD_i\,,\label{eq:19}\\
&E^{(T)}(v_i;\tD_i)\leq E^{(T)}(u;\tD_i)\,.\label{eq:20}
\end{flalign}
Note that for any spherical cylinder of the form 
 $$
\mathcal{C}=\Big\{
(r\sin\varphi\cos\theta,r\sin\varphi\sin\theta,\cos\varphi)\,;\,
r\in(R-H,R), \varphi\in[0,\rho/R),\theta\in[0,2\pi)\Big\}
$$
there is a $C^1$ diffeomorphism
$\Psi:D_\rho\times(0,H)\to\mathcal{C}$ given by
\begin{equation*}
  \Psi(y_1,y_2,y_3)=\Big(\frac{y_3+R-H}{R}\Big)\big(y_1,y_2,\sqrt{R^2-y_1^2-y_2^2}\big)\,.
\end{equation*}
A direct computation yields that
\begin{equation}
  \label{eq:6}
  D\Psi=\text{Id}+O(\rho/R)+O(H/R)~\text{ and }D(\Psi^{-1})=\text{Id}+O(\rho/R)+O(H/R)\,.
\end{equation}

Consider each spherical cylinder
$$
\mathcal{C}_i=\Big\{ ry\,;\,y\in \tD_i,\,r\in\big(\frac{R-R^\alpha}{R},1\big)\Big\}\,,\quad i=1,\ldots,k\,.
$$
By applying a rotation (sending $a_i$ to $\mathcal{N}$) and then the map $\Psi^{-1}$ we  find a
regular cylinder $D_{r_i}\times(0,R^\alpha)$ on which we perform the
construction of \rlemma{lem:cyl}. Going back to $\mathcal{C}_i$, we get, taking into account \eqref{eq:6}, 
 a map $U_i\in H^1(\mathcal{C}_i,\R^2)$ satisfying:
\begin{align}
  U_i(x)&=u(x), \quad x\in \tD_i\,,\label{eq:21}\\
 U_i(x)&=v_i\big(Rx/(R-R^\alpha)\big), \quad x\in \big((R-R^\alpha)/R\big)\tD_i\subset S_{R-R^\alpha}\,,\label{eq:22}\\
 U_i(x)&=u\big(Rx/|x|\big)=v_i\big(Rx/|x|\big),
         \quad \frac{Rx}{|x|}\in\partial \tD_i\,,\label{eq:23}
\intertext{and}
E(U_i;\mathcal{C}_i)&\leq C(R^\alpha+r_i^2/R^\alpha)\,E^{(T)}(u;\tD_i)\leq
            CR^\alpha\,E^{(T)}(u;\tD_i)\,.\label{eq:24}
\end{align}
Denoting by $\mathcal{P}_R(x)=Rx/|x|$ the radial projection on $S_R$, we finally
define $U$ in $B_R\setminus B_{R-R^\alpha}$ by
\begin{equation}
  \label{eq:25}
 U(x)=\begin{cases}
 U_i(x) & \text{ if }\mathcal{P}_R(x)\in \tD_i \text{ for some }i\\
  u(\mathcal{P}_R(x)) & \text{otherwise}
\end{cases}
\,.
\end{equation}
We clearly have
\begin{equation}
  \label{eq:26}
\begin{aligned}
  U(x)&=u(x) \text{ on }S_R\\
  V(x)&:=U\big|_{S_{R-R^\alpha}}(x) \text{ satisfies
  }7/8\leq |V(x)|\leq K~\text{ on
  }S_{R-R^\alpha}
\end{aligned}
\,.
\end{equation}
Moreover,
\begin{align}
\label{eq:27}
E^{(T)}(V;S_{R-R^\alpha})&\leq E^{(T)}(u;S_{R})\,,\\
\label{eq:28}
E(U;B_R\setminus B_{R-R^\alpha})&\leq CR^\alpha\ln R\,.
\end{align} 

Indeed, \eqref{eq:27} follows by summing up the contribution of each
$\tD_i$ in  \eqref{eq:20} and using \eqref{eq:25}. The inequality
 \eqref{eq:28} follows from \eqref{eq:24}--\eqref{eq:25} and \eqref{eq:ineq-prop}.\\[2mm]
\noindent{\underline{Step 3:}} Extension in $
B_{R-R^\alpha}$.\\[2mm]
 On $S_{R-R^\alpha}$ we may write
 \begin{equation}
   U(x)=\tilde\rho(x)e^{i\tilde\varphi(x)}\,,
 \end{equation}
 with $7/8\leq \tilde\rho\leq K$.
We first extend to $U(x)=\rho(x)e^{i\varphi(x)}$ in
$A(R):=B_{R-R^\alpha}\setminus \overline{B}_{R-2R^\alpha}$ by setting:
\begin{align}
  \rho(x)&=\frac{r-(R-2R^\alpha)}{R^\alpha}\tilde\rho\big((R-R^\alpha)x/|x|\big)+\frac{(R-R^\alpha)-r}{R^\alpha}\label{eq:54}\\
\varphi(x)&=\varphi\big((R-R^\alpha)x/|x|\big)
\,,\label{eq:55}
\end{align}
where
$r=|x|\in(R-2R^\alpha,R-R^\alpha)$. A direct computation gives:
\begin{align}
\label{eq:29}
  \int_{A(R)}\Big|\frac{\partial\rho}{\partial r}\Big|^2&\leq
  \frac{C}{R^\alpha}\ln R\,,\\
 \label{eq:51}\int_{A(R)}\big|\nabla_T\rho\big|^2&\leq CR^\alpha\ln R\,,\\
 \label{eq:52}\int_{A(R)}(1-\rho^2)^2&\leq CR^\alpha\ln R\,,\\
\label{eq:53}\int_{A(R)}|\nabla\varphi|^2&\leq CR^\alpha\ln R\,.
\end{align}
 From \eqref{eq:29}--\eqref{eq:53} it follows that
 \begin{equation}
   \label{eq:30}
 E(U;A(R))\leq CR^\alpha\ln R\,.
 \end{equation}
Moreover, by \eqref{eq:55} and \eqref{eq:25}--\eqref{eq:26} we have
\begin{equation}
  \label{eq:56}
 \int_{S_{R-2R^\alpha}}|\nabla\varphi|^2= 
 \int_{S_{R-R^\alpha}}|\nabla\varphi|^2\leq 2(8/7)^2E^{(T)}(u;S_R)\,.
\end{equation}

Finally, denoting by $\Phi$ the harmonic extension of $\varphi$ to
$B_{R-2R^\alpha}$, we set
\begin{equation*}
  U(x)=e^{i\Phi(x)}~\text{ in }B_{R-2R^\alpha}\,.
\end{equation*}
Combining \eqref{eq:harmonic},\eqref{eq:56}, \eqref{eq:28} and
\eqref{eq:30} we obtain
\begin{equation}
  \label{eq:57}E(U;B_R)\leq 
 (1/2)(8/7)^2(R-2R^\alpha)E^{(T)}(u;S_R)+CR^\alpha\ln R\,,
\end{equation}
 that clearly implies \eqref{eq:prop} for $R$ large, since $E(u;B_R)\leq E(U;B_R)$.
\end{proof}

\section{Proof of the main results}
We begin with the proof of \rth{th:eta}. It will be
more convenient to prove instead the theorem under the following equivalent
formulation, obtained by rescaling:
\begin{theorem}
  \label{th:main}
For every $\gamma\in(0,2\pi)$ and $\lambda>0$ there exists
$R_1(\gamma,\lambda)$ such that for every $u$ which is a minimizer of $E$ on $B_R$,
with $R\ge R_1(\gamma,\lambda)$, satisfying $E(u;B_R)\leq \gamma R\ln R$
there holds
\begin{equation}
  \label{eq:59}
  \big||u(0)|-1\big|\leq \lambda\,.
\end{equation}
\end{theorem}

\begin{proof}[Proof of \rth{th:main}]  We will use the shorthand
  $E(R)$ for $E(u;B_R)$ and also
  \begin{align*}
    E'(R)&:=\frac{d}{dR}E(R)=\int_{S_R}\frac{1}{2}|\nabla
                            u|^2+\frac{1}{4}(1-|u|^2)^2\,,\\
  e_T(R)&:=E^{(T)}(u;S_R)=\int_{S_R}\frac{1}{2}|\nabla_T
                            u|^2+\frac{1}{4}(1-|u|^2)^2\,.
  \end{align*}
The value of $R_1=R_1(\gamma,\lambda)$ will be determined later. For the moment we
will assume its value is known and in the course of the proof we shall
obtain some constraints that will allow us to detrmine its value
definitively. We first fix $\varepsilon$ satisfying
\begin{equation}
  \label{eq:63}
  \gamma+3\varepsilon<2\pi\,.
\end{equation}
Next we take  any $\beta\in(0,1)$, e.g., $\beta=1/2$. Since
\begin{equation*}
  \int_{R^\beta}^RE'(r)\,dr\leq E(R)<\gamma R\ln
  R=\gamma\int_{R^\beta}^R (\ln r+1)\,dr+\gamma R^\beta\ln(R^\beta)\,,
\end{equation*}
 there exists $\rho_1\in[R^\beta,R]$ such that
 \begin{equation}
   \label{eq:61}
   E'(\rho_1)\leq \gamma\Big(\ln
   \rho_1+1+\frac{R^\beta\ln(R^\beta)}{R-R^\beta}\Big)\leq
   (\gamma+\varepsilon)\ln \rho_1\,,
 \end{equation}
 if $R_1$ is chosen large enough.
Since $e_T(\rho_1)\leq E'(\rho_1)\leq (\gamma+\varepsilon)\ln \rho_1$,
we may apply \rprop{prop} (with $\gamma+\varepsilon$ playing the role
of $\gamma$) to get
\begin{equation}
  \label{eq:62}
  E(\rho_1)\leq \sigma\rho_1e_T(\rho_1)+C\rho_1^\alpha\ln\rho_1\leq \sigma(\gamma+2\varepsilon)\rho_1\ln\rho_1\,,
\end{equation}
provided $R_1$ is large enough. Next we fix $\delta\in(0,1)$ satisfying
\begin{equation}
  \label{eq:64}
  \gamma_0:=\frac{(\gamma+3\varepsilon)}{\delta}<2\pi\,.
\end{equation}
By \eqref{eq:62} we have, for every $r\in[\delta\rho_1,\rho_1]$,
\begin{equation}
  \label{eq:65}
  E(r)\leq
  E(\rho_1)\leq\frac{\sigma(\gamma+2\varepsilon)}{\delta}\,\delta\rho_1\ln\rho_1
<  (\sigma\gamma_0)\delta\rho_1\,\ln(\delta\rho_1)\leq (\sigma\gamma_0)\,r\ln r\,,
\end{equation}
again, if $R_1$ is chosen large enough. Let $M>0$ denote a constant
such that for $R\geq M$ all the inequalities,
\eqref{eq:61},\eqref{eq:62} and \eqref{eq:65} hold true.
Set
\begin{equation}
  \label{eq:66}
  \rho_2:=\inf\{R\in[\wtRz,\rho_1]\,:\,E(s)<(\sigma\gamma_0)\,s\ln
  s,\,\forall s\in[R,\rho_1]\}\,,
\end{equation}
 where $\wtRz=\widetilde R_0(\gamma_0,\sigma)\geq
 r_0(\gamma_0)$ (see \rprop{prop}) will be
 determined below.
Thanks to \eqref{eq:65} we have $\rho_2\leq\delta\rho_1$.
For $r\in[\rho_2,\rho_1]$ we have, either
\begin{equation}
  \label{eq:67}
  E'(r)\geq \frac{E(r)}{\sigma r}\,,
\end{equation}
or
\begin{equation}
  \label{eq:68}
  E'(r)<\frac{E(r)}{\sigma r}\,.
\end{equation}
 In the latter case, when \eqref{eq:68} holds, we have by \eqref{eq:66},
 \begin{equation*}
   e_T(r)\leq E'(r)<\frac{E(r)}{\sigma r}\leq\gamma_0 \ln r\,.
 \end{equation*}
Applying \rprop{prop} yields
\begin{equation}
  \label{eq:70}
  E'(r) \ge e_T(r)\geq \frac{E(r) - Cr^\alpha\ln r}{\sigma r}.
\end{equation}
The first possibility \eqref{eq:67} clearly implies
\eqref{eq:70}. Hence,   for {\em every} $r\in[\rho_2,\rho_1]$, \eqref{eq:70} holds. 
We rewrite \eqref{eq:70} as
\begin{equation}
  \label{eq:71}
  (r^{-1/\sigma}E(r))'\ge -\frac C\sigma r^{\alpha - 1 - 1/\sigma}\ln r\,.
\end{equation}
Integrating \eqref{eq:71} for 
 $r\in[\rho_2,\rho_1]$ gives
 \begin{equation}
\label{eq:73}
   \frac{E(\rho_1)}{\rho_1^{1/\sigma}}-\frac{E(\rho_2)}{\rho_2^{1/\sigma}}\geq 
-\frac C\sigma\left( \frac{{\rho_1}^{\alpha - 1/\sigma}}{\alpha -
    1/\sigma}\ln \rho_1 - 
\frac{{\rho_1}^{\alpha - 1/\sigma}}{(\alpha - 1/\sigma)^2} - \frac{{\rho_2}^{\alpha - 1/\sigma}}{\alpha - 1/\sigma}\ln \rho_2 + \frac{{\rho_2}^{\alpha - 1/\sigma}}{(\alpha - 1/\sigma)^2}\right)\,,
 \end{equation}
whence by \eqref{eq:65},
\begin{equation}
  \label{eq:72}
\begin{aligned}
  E(\rho_2)\leq &\bigl(\frac{\rho_2}{\rho_1}\bigr)^{1/\sigma}
  E(\rho_1)+C(\alpha,\sigma)\rho_2^\alpha\ln\rho_2
    \leq
    \bigl(\frac{\rho_2}{\rho_1}\bigr)^{1/\sigma}\sigma\gamma_0\,(\delta\rho_1)\ln(\delta\rho_1)+C(\alpha,\sigma)\rho_2^\alpha\ln\rho_2\\
&= (\sigma\gamma_0)\delta^{1/\sigma}\rho_2^{1/\sigma}
(\delta\rho_1)^{1-1/\sigma}\ln(\delta\rho_1)
+C(\alpha,\sigma)\rho_2^\alpha\ln\rho_2\,.
\end{aligned}
\end{equation}
Using $\rho_2\leq\delta\rho_1$ and the fact that the function
$t\mapsto (\ln t)t^{1-1/\sigma}$ is decreasing for $t\geq c_0$, we obtain from \eqref{eq:72} that 
\begin{equation}
  \label{eq:74}
\begin{aligned}
  E(\rho_2)&\leq
  (\sigma\gamma_0)\delta^{1/\sigma}\rho_2^{1/\sigma}\cdot\rho_2^{1-1/\sigma}\ln \rho_2
+C(\alpha,\sigma)\rho_2^\alpha\ln\rho_2\\
&= (\sigma\gamma_0)\delta^{1/\sigma}\rho_2\ln \rho_2+C(\alpha,\sigma)\rho_2^\alpha\ln\rho_2\,.
\end{aligned}
\end{equation}
Let $r_1$ denote the value of $\rho_2\neq 1$ for which
equality holds between the right hand sides of \eqref{eq:74} and
\eqref{eq:66} (for $s=\rho_2$). That is, $r_1$
satisfies  the equality
\begin{equation}
\label{eq:31}
  (\sigma\gamma_0)\,r_1=
  (\sigma\gamma_0)\delta^{1/\sigma}r_1+C(\alpha,\sigma)r_1^\alpha\,.
\end{equation}
A simple computation gives
\begin{equation}
  \label{eq:4}
  r_1=\left(\frac{C(\alpha,\sigma)}{\sigma\gamma_0(1-\delta^{1/\sigma})}\right)^{\frac{1}{1-\alpha}}\,,
\end{equation}
 and we may indeed assume that $r_1>1$ by replacing,
 if necessary, $C(\alpha,\sigma)$ by a larger constant.
Note that
\begin{equation}
\label{eq:82}
  (\sigma\gamma_0)\,\rho\ln\rho>
 (\sigma\gamma_0)\delta^{1/\sigma}\rho\ln
 \rho+C(\alpha,\sigma)\rho^\alpha\ln\rho\quad\text{ for }\rho>r_1\,.
\end{equation}
 Next we claim that if we take in \eqref{eq:66} $\wtRz=\widetilde
R_0(\gamma)$ satisfying 
\begin{equation}
  \label{eq:87}
  \wtRz\geq \max\{r_0(\gamma_0),r_1,M\}\,,
\end{equation}
 then necessarily
\begin{equation}
  \label{eq:75}
  \rho_2=\wtRz\,.
\end{equation}
Indeed, otherwise $\rho_2>\widetilde R_0$
and by \eqref{eq:66} and \eqref{eq:82} we must have
\begin{equation}
  \label{eq:76}
  E(\rho_2)=(\sigma\gamma_0)\rho_2\ln \rho_2>(\sigma\gamma_0)\delta^{1/\sigma}\rho_2\ln
 \rho_2+C(\alpha,\sigma)\rho_2^\alpha\ln\rho_2\,,
\end{equation}
which contradicts the bound in \eqref{eq:74}. 

In view of \eqref{eq:75} we know that \eqref{eq:71} holds for all $r\in
[\wtRz,\rho_1]$, and integration over this interval yields,
as in \eqref{eq:73}--\eqref{eq:72},
\begin{equation}
  \label{eq:77}
\begin{aligned}
  E(\wtRz)&\leq
  \sigma\gamma_0\delta\rho_1^{1-1/\sigma}{\wtRz}^{1/\sigma}\ln(\delta\rho_1)+C(\alpha,\sigma){\wtRz}^\alpha\ln\wtRz\\
&= \sigma\gamma_0\delta{\wtRz}^\alpha\Big({\wtRz}^{1/\sigma-\alpha}/\rho_1^{1/\sigma-1}\Big) \ln(\delta\rho_1)+C(\alpha,\sigma){\wtRz}^\alpha\ln\wtRz\,.
\end{aligned}
\end{equation}
Let $\wtRo=\wtRo(\wtRz)$ be determined by the equation
\begin{equation}
  \label{eq:877}
  \wtRz^{1/\sigma-\alpha}=\wtRo^{\beta(1/\sigma-1)}/\ln (\wtRo^\beta)\;(\leq \rho_1^{1/\sigma-1}/\ln\rho_1)\,.
\end{equation}
It follows from  \eqref{eq:77} that for every $R_1\geq \wtRo(\wtRz)$,
where $\wtRz$ satisfies \eqref{eq:87}, we have (under the assumption
of the Theorem, i.e., $E(R)\leq \gamma R\ln R$ for some $R\geq R_1$):
\begin{equation}
  \label{eq:78}
  E(\wtRz)\leq
  \sigma\gamma_0\delta{\wtRz}^\alpha+C(\alpha,\sigma){\wtRz}^\alpha\ln\wtRz\leq \wtC{\wtRz}^\alpha\ln\wtRz\,,
\end{equation}
for some $\wtC$ (which actually depends only on $\gamma$).

Finally we turn to the proof of \eqref{eq:59}. It is clearly enough to consider $\lambda<2K$ (where $K$ is given by
\rlemma{lem:L-infty}). Looking for contradiction, assume that 
\begin{equation}
\label{eq:89}
  \big||u(0)|-1\big|>\lambda\,.
\end{equation} 
Then, by \rlemma{lem:L-infty},
\begin{equation}
\label{eq:79}
  \big||u(x)|-1\big|>\lambda/2\text{
    in }B_{\lambda/2K}\,.
\end{equation}
By \eqref{eq:78} and \eqref{eq:79} and \rlemma{lem:mono} we obtain
\begin{equation}
  \label{eq:80}
\frac{1}{4}\big(\frac{\lambda}{2}\big)^2\big(\frac{\lambda}{2K}\big)^3|B_1|\leq
E(u;B_{\lambda/2K})\leq E(u;B_1)\leq E(\wtRz)/\wtRz\leq \wtC{\wtRz}^{\alpha-1}\ln\wtRz\,.
\end{equation}
Let $T=T(\lambda,\alpha)$ be a large enough value of $\wtRz$ for which
\eqref{eq:80} is violated, i.e.,
\begin{equation}
  \label{eq:88}
  \frac{\lambda^5|B_1|}{128K^3}>\wtC T^{\alpha-1}\ln T.
\end{equation}
Therefore, taking $\wtRz=\wtRz(\gamma,\lambda)$ satisfying both \eqref{eq:87}
and $\wtRz\geq T$, and then setting  $R_1(\gamma,\lambda):=\wtRo(\wtRz)$ (see
\eqref{eq:877}),  we see that  \eqref{eq:80}, and hence also \eqref{eq:89}, cannot hold for $R\geq
R_1(\gamma,\lambda)$. 
 \end{proof}
 \begin{proof}[Proof of \rth{th:lm}]
   Fix any point $x\in\R^3$. For each $\lambda>0$ we may apply
   \rth{th:main} on $B_{R_n}(x)$ for an appropriate sequence $R_n\to\infty$  to conclude
   that $|u(x)|=1$. Hence $|u|\equiv 1$ in $\R^3$ and from
   \eqref{eq:81} we deduce that both components of
   $u=(u_1,u_2)$ are harmonic functions. Assuming without loss of
   generality that $u(0)=(1,0)$, we conclude from the maximum principle
   that $u_1\equiv 1$ and therefore $u\equiv (1,0)$ in $\R^3$.
 \end{proof}
 
 \section*{Appendix. Proof of \rlemma{lem:specialists}}
 \setcounter{equation}{0}
 \renewcommand{\theequation}{\Alph{section}.\arabic{equation}}
  \setcounter{section}{1}
  \renewcommand{\thelemma}{A}

 \begin{proof}
 We let $\varepsilon = 1/R$ and $v(x) = u(Rx)$. Then $v:S_1\to\R^2$ satisfies 
  \begin{equation}\label{bg}E_\varepsilon^{(T)}(v;S_1)\leq \gamma\ln\frac1\varepsilon\,.\end{equation}
  Following \cite{bouquin}, proof of Theorem 5.3, given $\delta>0$
  there exists a collection of disjoint spherical discs that we denote
  $\mathcal D_0$ such that each disc in the family has radius bounded
  below by $\Lambda\varepsilon$, such
  that $|v|> 1 - \delta$ on the complement of $\mathcal D_0$ and such
  that, denoting by $r_0$ the sum of the radii of the discs, we have
  \begin{equation}
    \label{eq:84}
    r_0 \le C\gamma|\ln\varepsilon| \frac{\varepsilon}{\delta^3}\,.
  \end{equation}
  It should be noted that in contrast with the situation considered in
  \cite{bouquin}, the restriction of $v$ to $S_1$ does not necessarily 
  satisfy the
  Ginzburg-Landau equation (which would involve the Laplace-Beltrami
  operator on $S_1$). However, \eqref{eq:58b} implies that the restriction of
  $v$ to $S_1$ satisfies an estimate of the form $\|\nabla v\|_{L^\infty}\leq
  C/\varepsilon$, which is what needed to construct $\mathcal D_0$
  using the method of \cite{bouquin}.
  Then we apply the ball-growth procedure as in
  \cite[Thm~4.2]{bouquin}. This yields, for any $t>0$, a family of spherical discs $\mathcal D(t)$ which is increasing and such that the sum of the radii of the discs in $\mathcal D(t)$ is $e^t r_0$. Moreover, denoting
  $$
 \mathcal F(x,r) = E_\varepsilon^{(T)}(v;\tD_r(x))
~\text{ and }~\mathcal{F}(\mathcal D(t))=\sum_{\tD_r(x)\in\mathcal D(t)}\mathcal F(x,r) \,,
 $$
  we deduce from \cite[Proposition~4.1]{bouquin} and \eqref{bg}
  that, for all $s>0$, 
  $$
\gamma|\ln\varepsilon| \ge \mathcal F(\mathcal D(s)) -  \mathcal
F(\mathcal D(0)) \ge \int_0^s \sum_{\tD_r(x)\in\mathcal D(t)}
r\frac{\partial\mathcal F}{\partial r} (x,r)\,dt\,.
$$
  Hence, by Fubini's theorem  there exists $t\in(0,s)$ such that $\mathcal D(t) = \{\tD_{\rho_i}= \tD_{\rho_i}(x_i), 1\le i\le k\}$ satisfies
 $$
\sum_{i=1}^k \rho_i\frac{\partial\mathcal F}{\partial r} (x_i,\rho_i) \le
 \frac{ \gamma|\ln\varepsilon|}s\,.
$$
 Choosing
 $ s = \frac{2\pi +\gamma}{4\pi} |\ln\varepsilon|$, 
 we find that, 
 \begin{equation}
\label{infpi}
 \sum_{i=1}^k \rho_i\frac{\partial\mathcal F}{\partial r} (x_i,\rho_i) \le
 2\pi\big(\frac{2\gamma}{\gamma+ 2\pi}\big)\,.
\end{equation}
 On the other hand, since $|v|> 1 - \delta$ on the complement of
 $\mathcal D_0$, we have for each $i$
 $$\frac{\partial\mathcal F}{\partial r} (x_i,\rho_i) =
 E_\varepsilon^{(T)}(v;\partial\tD_{\rho_i}) \ge \frac{\pi}{\sin \rho_i}
 \deg(v,\partial\tD_{\rho_i})^2 (1-\delta)^2
\geq
\frac{\pi}{\rho_i}
 \deg(v,\partial\tD_{\rho_i})^2 (1-\delta)^2 \,,
$$
 where we used the fact that the Euclidean radius of the circle
 $\partial\tD_{\rho_i}$ is $\sin \rho_i$.
 Substituting  in \eqref{infpi} yields
 \begin{equation}
   \label{eq:83}
   \sum_{i=1}^k \deg(v,\partial\tD_{\rho_i})^2 \le \frac2{(1-\delta)^2} \cdot \frac{2\gamma}{\gamma+ 2\pi} <2\,,
 \end{equation}
the last inequality on the R.H.S.~of \eqref{eq:83} holds for
$\delta>0$ small enough, 
  $\delta<\delta_0(\gamma)$. 
We thus fix any  $\delta\in\big(0,\min(1/8,\delta_0(\gamma))\big)$. Since
$\sum_{i=1}^k \deg(v,\partial\tD_{\rho_i}) =0$, it follows from
\eqref{eq:83} that $\deg(v,\partial\tD_{\rho_i}) =0$ for all $i$. Indeed,
otherwise we would have at least two nonzero degrees, clearly
violating  \eqref{eq:83}. 
 
 Now the sum of the radii of the discs in $\{\tD_{\rho_i}\}_{i=1}^k$ is
 $r = e^t r_0$ hence bounded by $e^s r_0$.  Using \eqref{eq:84} we
 conclude that for some $C>0$ depending on $\gamma$, 
 $$ r \le  C|\ln\varepsilon| \varepsilon \times \varepsilon^{-
   \frac{2\pi +\gamma}{4\pi}} =
 C|\ln\varepsilon| \varepsilon^{\frac{2\pi-\gamma}{4\pi}}\leq \varepsilon^{\tilde\alpha}\,,$$
 for every small enough $\varepsilon$, provided  $0<\tilde\alpha<
 \frac{2\pi-\gamma}{4\pi}$. We fix such a value of $\tilde \alpha$. 
 
 The family of spherical discs $\{\tD_{\rho_i}\}_{i=1}^k$ is therefore such that $|v|> 1/2$ outside the discs, their total radius is  less than $\varepsilon^{\tilde\alpha}$,  the winding number of $v$ on the boundary of each disc is zero, and such that for each $1\le i\le k$ we have using \eqref{infpi} that 
 $$ E_\varepsilon^{(T)}(v;\partial\tilde D_{\rho_i}) \le \frac{2\pi}{\rho_i} \frac{2\gamma}{\gamma+ 2\pi} <  \frac{2\pi}{\rho_i}.$$
 
 Rescaling back, we see easily that the spherical discs on $S_R$, $\tD_{r_i}^R(Rx_i):=R\tD_{\rho_i},i=1,\ldots,k$, satisfy
 all the assertions \eqref{eq:p1}--\eqref{eq:p5}  of \rlemma{lem:specialists} with
 $\alpha=1-\tilde\alpha$. 
  \end{proof}

\end{document}